\documentclass[a4paper]{article}

\usepackage[utf8]{inputenc}
\usepackage{lmodern}
\usepackage{array}
\usepackage{graphicx}
\usepackage{subfig}
\usepackage{amssymb, amsfonts, amsthm, amsmath}
\usepackage{enumerate}
\usepackage[english]{babel}
\usepackage[colorlinks]{hyperref}
\usepackage{tikz}

\theoremstyle{plain}
\newtheorem{theorem}{Theorem}

\newtheorem{lemma}[]{Lemma}

\theoremstyle{definition}

\theoremstyle{remark}
\newtheorem{remark}{Remark}

\newcommand{\R}{\mathbb{R}}

\newcommand{\ind}[1]{\mathbf{1}_{\left\{#1\right\}}}

\DeclareMathOperator{\E}{\mathbf{E}}
\renewcommand{\P}{\mathbf{P}}

\renewcommand{\epsilon}{\varepsilon}
\renewcommand{\phi}{\varphi}

\makeatletter \newcommand \listoftodos{\section*{Todo list}
  \@starttoc{tdo}} \newcommand\l@todo[2] {\par\noindent
  \textit{#2}, \parbox{10cm}{#1}\par} \makeatother

\newcommand{\sgn}{\mathrm{sgn}}

\usepackage{xcolor}

\title{Tight estimates of exit and containment probabilities for
  nonlinear stochastic systems}

\author{Quang-Cuong Pham\,$^{1,2}$, Bastien Mallein\,$^{3}$, Jean-Jacques Slotine\,$^{4}$}
\date{\small{
    $^1$Nanyang Technological University, Singapore\\%
    $^2$Eureka Robotics, Singapore\\%
    $^3$Universit\'e Paris 13, France\\%
    $^4$Massachusetts Institute of Technology, USA\\%
}
\vspace{0.5cm}
\today}

\begin{document}

\maketitle

\begin{abstract}
  Tight estimates on the exit/containment probabilities of stochastic
  processes are of particular importance in many control
  problems. Yet, estimating the exit/containment probabilities is
  non-trivial: even for linear systems (Ornstein-Uhlenbeck processes),
  the containment probability can be computed exactly for only some
  particular values of the system parameters. In this paper, we derive
  tight bounds on the containment probability for a class of nonlinear
  stochastic systems. The core idea is to compare the ``pull
  strength'' (how hard the deterministic part of the system dynamics
  pulls towards the origin) experienced by the nonlinear system at
  hand with that of a well-chosen process for which tight estimates of
  the containment probability are known or can be numerically obtained
  (e.g. an Ornstein-Uhlenbeck process). Specifically, the main
  technical contribution of this paper is to define a suitable
  dominance relationship between the pull strengths of two systems and
  to prove that this dominance relationship implies an order
  relationship between their containment probabilities. We also
  discuss the link with contraction theory and suggest some examples
  of applications.
\end{abstract}

\section{Introduction}

Consider a nonlinear, multi-dimensional, Stochastic Differential
Equation (SDE) of the form
\[
dX_t = f(X_t)dt + \sigma dB_t,
\]
where $f$ a smooth function and $\sigma$ a positive constant.

Given a ball of radius $R$ and a time instant $T$, the \emph{exit}
probability from the ball by time $T$ is defined
as~\cite{kushner1967stochastic}
\[
P_\mathrm{exit}:=\P(\sup_{t\leq T} \|X_t\| > R).
\]
Equivalently, one may consider the \emph{containment} probability, which
is 1-$P_\mathrm{exit}$, or, in other words
\[
P_\mathrm{cont}:=\P(\sup_{t\leq T} \|X_t\| \leq R).
\]

The exit/containment probabilities are of particular importance in
many control problems, including tracking,
filtering~\cite{kushner1967stochastic}, optical
manipulation~\cite{yan2016stochastic}, etc. The main reason is that,
in such applications, the validity of the system description by the
SDE at hand is guaranteed to be valid \emph{only} within some region
of space, for example, a ball of radius $R$ -- once the system exits
from the validity region, nothing can be said anymore about it, see
Fig.~\ref{fig:laser} for an illustration. Tight estimates of the
exit/containment probability are therefore crucial: one can then
reason on the system behavior \emph{conditioned} on the event that the
system is contained within the validity region \emph{at all time} up
to $T$. Note that, stochastic stability \emph{in the mean-square
  sense} and associated probability estimates based on
\emph{concentration inequalities}, which are more commonly found in
the literature, cannot account for such observation, as further
detailed in Section~\ref{sec:applications}.

\begin{figure}[htp]
  \centering
  \includegraphics[width=10cm]{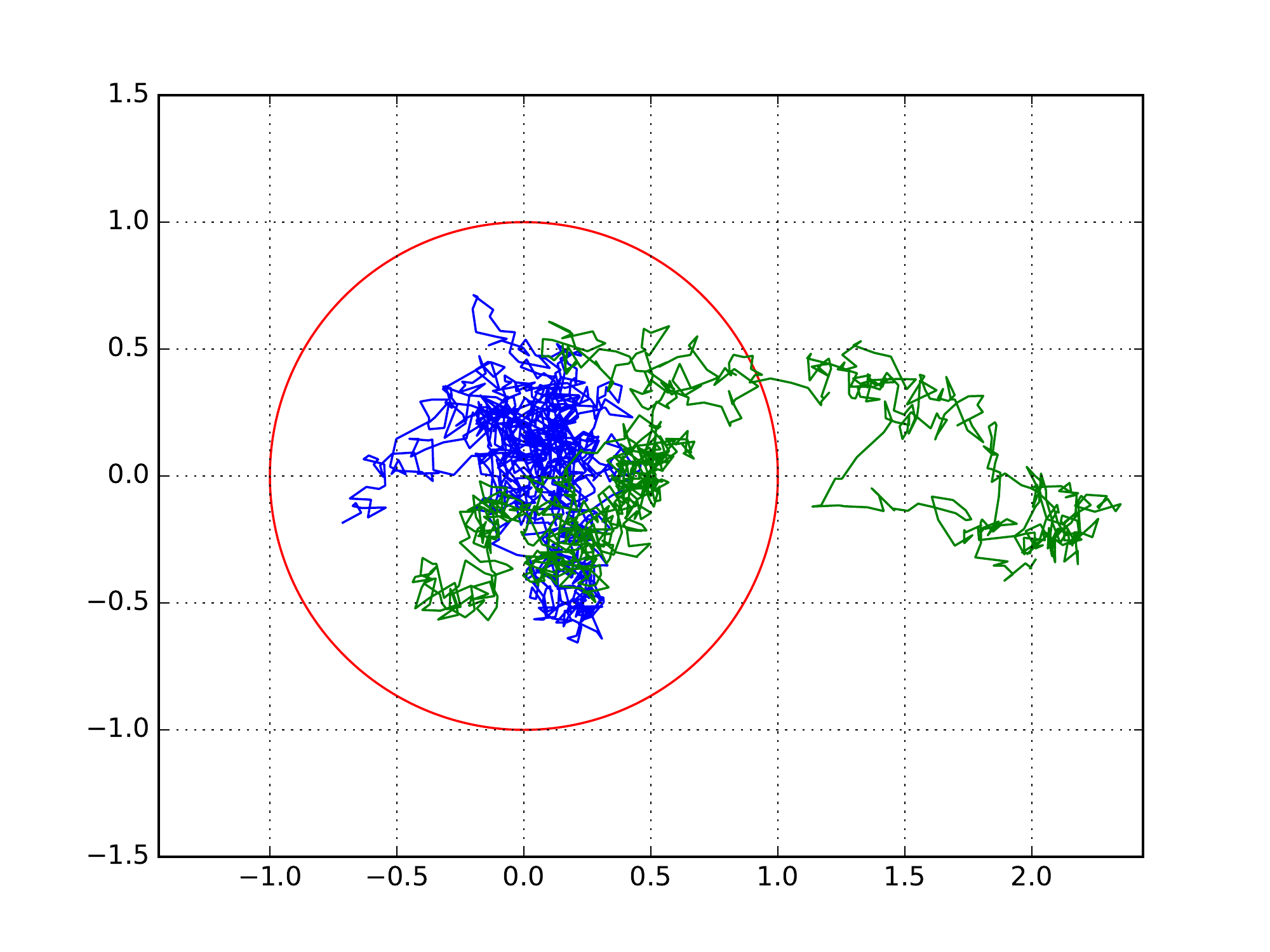}
  \caption{Simulation of a particle trapped by a laser. Within the
    laser beam (red circle), the particle is subject to two phenomena:
    (i) a trapping force that pulls the particle towards the center of
    the beam and (ii) random Brownian perturbations. Outside of the
    laser beam, the trapping force vanishes very quickly and can be
    considered as zero~\cite{yan2016stochastic}. Two sample
    trajectories for $t\in[0,T]$ are shown in blue and green. The
    green trajectory leaves the beam at some time $t<T$, after which
    one has essentially lost all control authority over it. It is
    therefore crucial, in any probability calculations, to consider
    the event that the particle is contained within the beam for all
    $t\in[0,T]$.}
  \label{fig:laser}
\end{figure}

Yet, estimating the exit/containment probability is non-trivial: even
for linear, uni-dimensional, systems of the following form (also known
as Ornstein-Uhlenbeck processes)
\[
dX_t = -\lambda X_t dt + \sigma dB_t,
\]
the containment probability can be computed exactly for only some
particular values of $\lambda, \sigma, R, T$~\cite{finch2004ornstein}.
Kushner's classic book on stochastic
control~\cite{kushner1967stochastic} provides some bounds on the
containment probability for nonlinear systems (within the topic of
``finite-time stability'') but, as we shall see in
Section~\ref{sec:loose}, those bounds are too loose for many practical
applications.

In this paper, we derive tight bounds on the containment probability
for a class of nonlinear stochastic systems. The core idea is to
compare the ``pull strength'' (how hard the deterministic part of the
system dynamics pulls towards the origin) experienced by the nonlinear
system at hand with that of a well-chosen process for which tight
estimates of the containment probability are known or can be
numerically obtained (e.g. an Ornstein-Uhlenbeck process). However, a
stronger pull everywhere does \emph{not always} imply a larger
containment probability, as made clear in
Section~\ref{sec:counter-ex}. The main technical contribution of this
paper is thus to define a suitable dominance relationship between the
pull strengths of two systems and to prove that this dominance
relationship implies an order relationship between their containment
probabilities.

The remainder of this paper is organized as follows.
Section~\ref{sec:context} provides the theoretical and practical
contexts of the problem at hand. Section~\ref{sec:main} presents the
main comparison results in dimensions $d=1$ and $d\geq 2$.
Section~\ref{sec:contraction} examines these results in the context of
contraction theory~\cite{lohmiller1998contraction}. Finally,
Section~\ref{sec:conclusion} concludes by sketching future research
directions.

\section{Theoretical and practical contexts}
\label{sec:context}

\subsection{Containment probability vs mean-square stability}
\label{sec:applications}

The stability of nonlinear stochastic systems is an active research
area with important applications ranging from observer and
controller design for nonlinear noisy
systems~\cite{karvonen2018stability}, to synchronization in networks
of noisy
oscillators~\cite{pham2009contraction,tabareau2010synchronization},
etc.

Most of the existing stochastic stability results are in the
``mean-square'' sense, typically of the
form~\cite{pham2009contraction,karvonen2018stability}
\begin{equation}
  \label{eq:l2}
  \forall t\geq 0,\ \E(\|X_t\|^2) \leq C_1 + C_2 e^{-2\lambda t},  
\end{equation}
where $C_1,C_2$ are two positive constants. While such results are
certainly useful to understand the system behavior \emph{on average},
they are not relevant when it comes to behaviors that depend on
\emph{individual} system trajectories. Consider for instance a
microscopic particle that is trapped by a laser
tweezer~\cite{ashkin1986observation} and subject to Brownian
perturbations~\cite{yan2016stochastic}. The motion of the particle is
well described by a Stochastic Differential Equation (SDE), but the
validity of that description breaks down when the particle escapes
from the laser trapping region, see Fig.~\ref{fig:laser}. Mean-square
stability results that do not account for this phenomenon will not
provide an accurate understanding of the system.

Note that it is possible to use concentration inequalities to derive,
from mean-square bounds~\eqref{eq:l2}, probabilities of the form
\[
\P\left(\|X_t\|^2 \geq (C_1 + C_2 e^{-2\lambda
    t})\beta(\delta)\right) \leq e^{-\delta},
\]
see e.g.~\cite{karvonen2018stability}. However, to be relevant, such
probabilities, as the mean-square bounds mentioned previously, must
be conditioned upon the containment event.

It is therefore crucial to consider the \emph{containment} probability
of the form
\[
  \P(\sup_{t\leq T} \|X_t\|^2 \leq R^2) \geq 1-\epsilon,
\]
where $R$ is the laser trapping radius and $\epsilon$ a small
constant. One can then reason on the particle behavior
\emph{conditioned} upon the above probability that the particle
remains confined within the trapping region \emph{at all time} until
time $T$.

Another important example concerns the analysis of the Extended Kalman
Filter (EKF)~\cite{karvonen2018stability}. The SDE describing the EKF
is obtained by linearizing the system dynamics around a reference
trajectory, and is therefore guaranteed to be valid only within some
radius $R$ around that trajectory. One thus needs to condition upon
the probability that the system trajectory remains within a radius $R$
from the reference trajectory up to some time horizon $T$. The reader
is referred to Chapter~III of~\cite{kushner1967stochastic} for an
extensive discussion of the relative merits of mean-square stability
versus exit/containment probabilities in control theory.



\subsection{Looseness of existing containment probability estimates}
\label{sec:loose}

In the literature, the main result on exit probability for nonlinear
systems was derived in Chapter~III of Kushner's
book~\cite{kushner1967stochastic}, based on a stochastic Lyapunov
analysis. However, the bound provided is too loose to be useful in
many practical applications. Consider again the laser trapping
application~\cite{yan2016stochastic}. The estimates of the containment
probability can be used to calculate the maximum velocity of the laser
beam such that the particle remains trapped with high
probability. In~\cite{yan2016stochastic}, it was shown that using
Kushner's estimates yields recommended maximum beam velocities that
are significantly lower than velocities experimentally found to be
safe.

To get a more precise idea, consider the following simple linear,
uni-dimensional, stochastic system (an Ornstein-Uhlenbeck process)
\begin{equation}
  \label{eq:ou}
  dX_t = -X_t dt + \sqrt{2}dB_t, \ X_0=0.
\end{equation}
The bound given by Kushner~\cite{kushner1967stochastic} would read
\begin{equation}
  \label{eq:bound-kushner}
  \P(\sup_{t\leq T} |X_t| \leq R) \geq e^{-\mu_K(R) T},\ \textrm{where}
\end{equation}
\begin{equation}
  \label{eq:decay-kushner}
  \mu_K(R) = \frac{2}{R^2}. 
\end{equation}

On the other hand, a direct, but more technically challenging,
analysis of system~\eqref{eq:ou}, gives the following bound~(see
\cite{finch2004ornstein})
\begin{equation}
  \label{eq:bound-direct}
  \P(\sup_{t\leq T} |X_t| \leq R) \sim_{T\to\infty} e^{-\mu_D(R) T},
\end{equation}
where $\mu_D(R)$ is the smallest value $\nu$ such that the
Sturm-Liouville system
\[
  \begin{cases}
    y''(x) - x y'(x) = -\nu y(x)\\
    y(-R) = y(R) =0
  \end{cases}
\]
has non-null solutions.
Reference~\cite{finch2004ornstein} then gives some values of
$\mu_D(R)$ for $0.7 \leq R \leq 3$, as well as the asymptotics
\begin{equation}
  \label{eq:decay-direct}
  \mu_D(R) \sim_{R\to\infty} \frac{R}{\sqrt{2\pi}}e^{-\frac{R^2}{2}}.
\end{equation}

Note that the asymptotics of \eqref{eq:bound-direct} and
\eqref{eq:decay-direct} yield very good approximations as soon as
$T\geq 5$ and $R\geq 3$.

Fig.~\ref{fig:comp} compares the decay rate $\mu_K$ given by
Kushner and the decay rate $\mu_D$ obtained by direct analysis. One
can observe that $\mu_K$ is tight for $R\leq 2$, but becomes very
loose for large values of $R$.

\begin{figure}[htp]
\centering
\includegraphics[width=10cm]{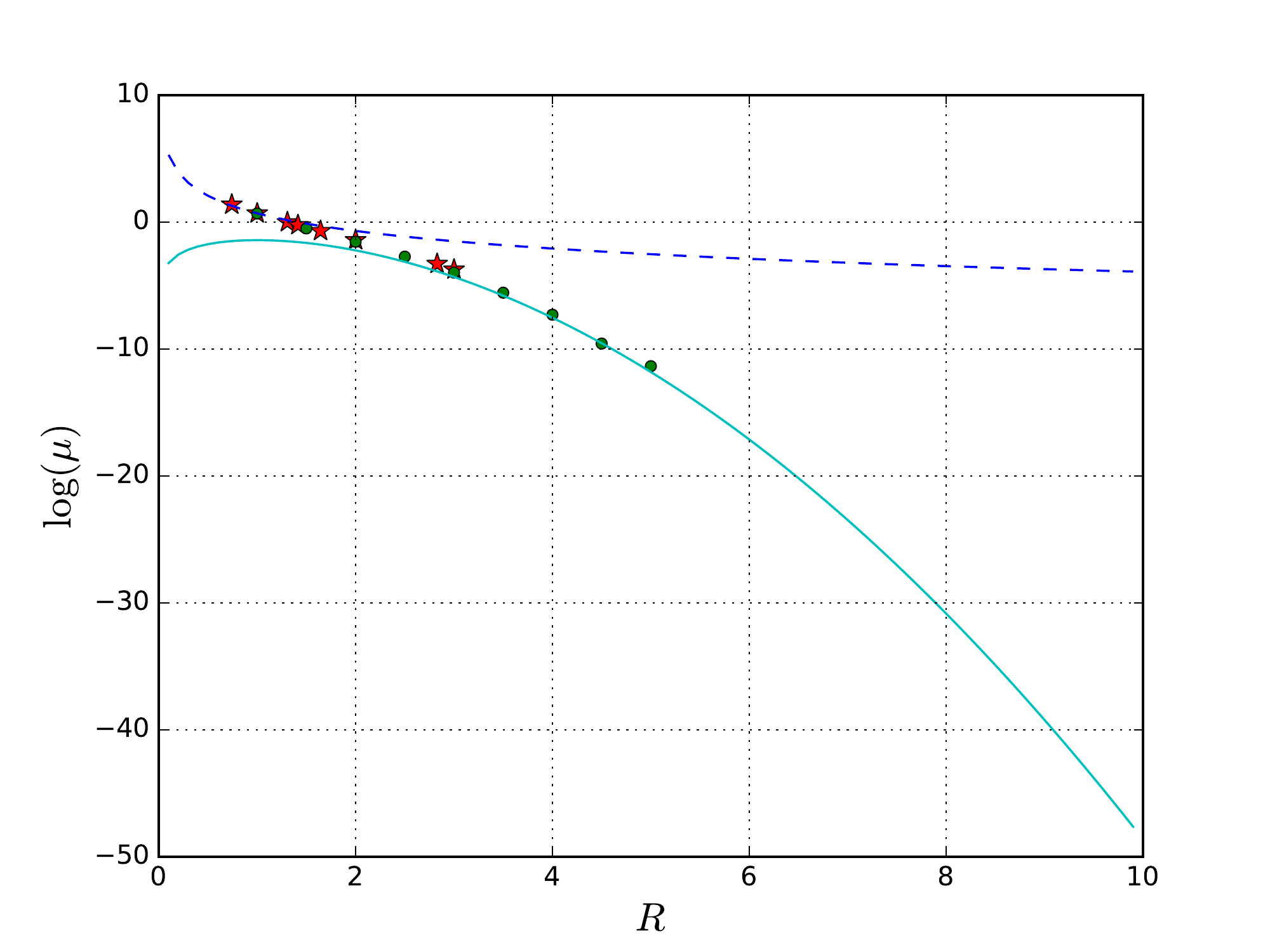}
\caption{Comparison (in log scale) of the decay rate $\mu_K$ given by
  Kushner (dashed blue) and the decay rate $\mu_D$ obtained by direct
  analysis of the Ornstein-Uhlenbeck process [red stars: exact values;
  green circles: values obtained by Monte Carlo simulations with
  $T=5$, cyan plain line: asymptotics for large $R$ given
  by~\eqref{eq:decay-direct}]. Note that the asymptotics yield very
  good approximations as soon as $R\geq 3$.}
\label{fig:comp}
\end{figure}

In~\cite{yan2016stochastic}, it was shown, in a physical laser
trapping experiment, that, in contrast with Kushner's
estimates~\eqref{eq:bound-kushner}, the direct
estimates~\eqref{eq:bound-direct} yields recommended maximum beam
velocities that agree extremely well with the velocities
experimentally found to be safe.

\subsection{A stronger pull does not always imply a larger containment
  probability}
\label{sec:counter-ex}

Consider two uni-dimensional systems (one can get rid of $\sigma$ by
adequate normalization)
\begin{eqnarray*}
dX_t &=& f(X_t)dt + dB_t,\\
dY_t &=& g(Y_t)dt + dB'_t.
\end{eqnarray*}

Suppose that, \emph{everywhere}, $X_t$ experiences a stronger pull
than $Y_t$ towards zero, that is,
\begin{equation}
  \label{eq:assum-orig}
  \forall x\in\mathbb{R},\ \sgn(x)f(x) \leq \sgn(x) g(x),  
\end{equation}
where the sign function defined by
\begin{eqnarray}
  \label{eq:sgn}
  \mathrm{if}\ x>0,&\ & \sgn(x) := 1,\nonumber\\  
  \mathrm{if}\ x=0,&\ & \sgn(x) := 0,\nonumber\\  
  \mathrm{if}\ x<0,&\ & \sgn(x) := -1\nonumber.  
\end{eqnarray}

Then, one would like to say that, for all $T, R$,
\begin{equation}
  \label{eq:comp}
  \P(\sup_{t\leq T} |X_t| \leq R) \geq \P(\sup_{t\leq T} |Y_t| \leq R).
\end{equation}
However, this is \emph{not} always true, as shown by the following
counter-example.




For a given $\lambda \in \R$, denote by $X^{\lambda}$ the
strong solution of the SDE
\[
  dX_t = -X_t \ind{X_t>0} dt - \lambda X_t \ind{X_t < 0} + dB_t,
\]
i.e. $X^{\lambda}$ is similar to an Ornstein-Uhlenbeck process with
pull 1 on the right half-line, and with pull strength $\lambda$ on the left
half-line. Note that, if $\lambda>\lambda'$, then $X^\lambda$ is
subject to a pull stronger or equal to that of $X^{\lambda'}$
everywhere.

Given $R,T$, consider the containment probability
\[
P(\lambda) := \P\left( \sup_{t \leq T} |X^\lambda_t| \leq R \right).
\]

Fig.~\ref{fig:contre-ex} shows the values of $P(\lambda)$ for
$\lambda\geq 1$, $R=0.5$, $T=1$. One can observe that $P(\lambda)$ is
non-monotonic: it increases on the right of $\lambda=1$, reaches a
maximum at $\lambda\simeq 20$, then decreases towards $P(1)$ when
$\lambda\to\infty$.

\begin{figure}[htp]
\centering
\includegraphics[width=10cm]{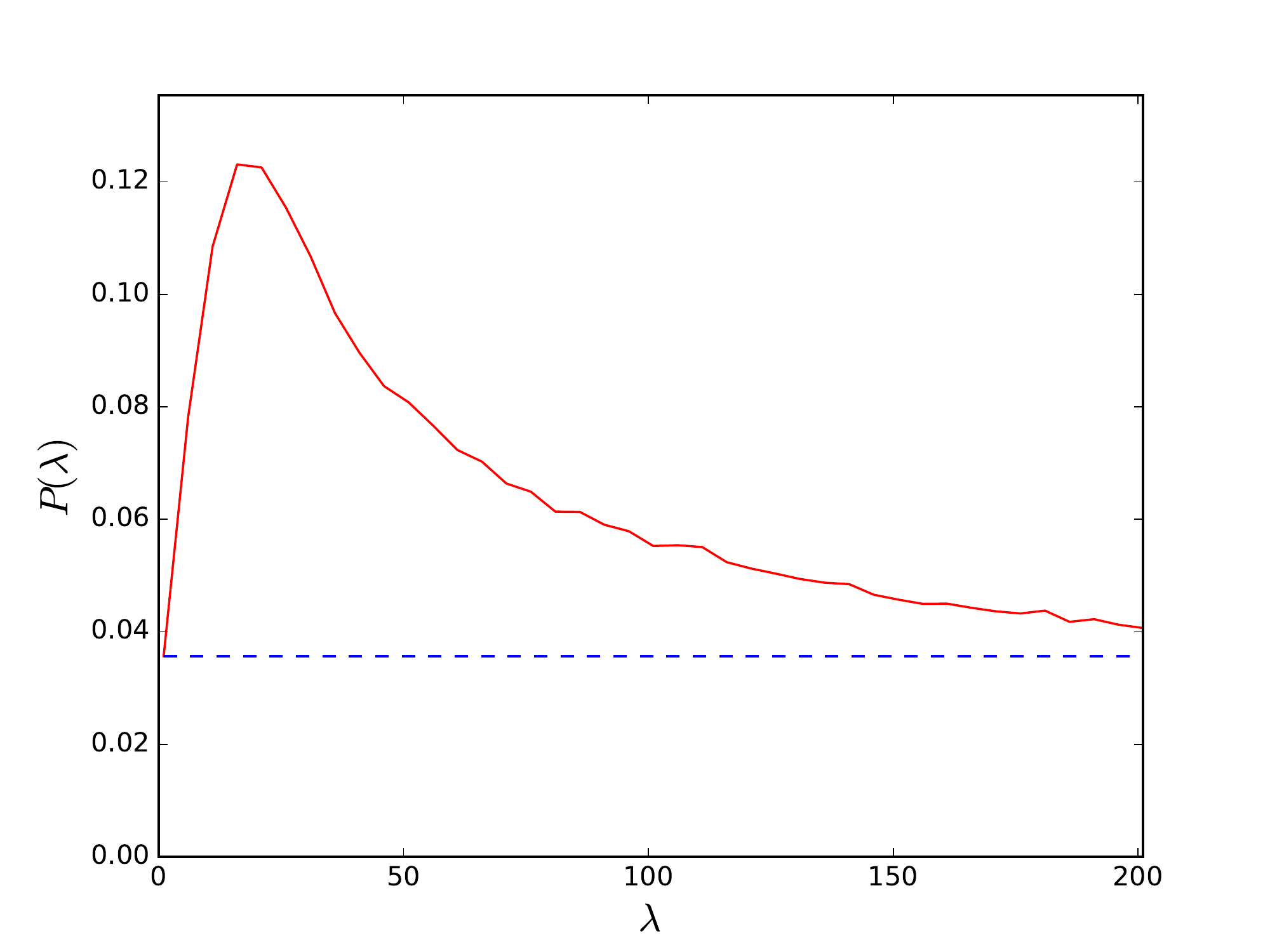}
\caption{A stronger pull does not always imply a larger containment
  probability. The red plain line shows the containment probability
  $P(\lambda)$ for $\lambda\geq 1$ with $R=0.5, T=1$ (obtained by
  Monte Carlo simulation). One can note that the containment
  probability increases with the pull strength $\lambda$ only until
  $\lambda\simeq 20$, then it decreases towards $P(1)$ when
  $\lambda\to\infty$ (dashed blue line).}
\label{fig:contre-ex}
\end{figure}

Intuitively, increasing the strength of the pull on the left half-line
has two opposite effects:
\begin{enumerate}
\item A stronger pull on the left half-line makes exits by the left
  boundary more ``difficult'', thereby contributing positively to the
  containment probability;
\item But, at the same time, increasing the pull strength
  asymmetrically might ``chase'' the diffusion from an area where it
  is well-controlled towards an area where the control is not as good,
  which can, in turn, help the process escape. In the limit
  $\lambda \to \infty$, the left half-line acts as a solid
  wall. Therefore, $\lim_{\lambda \to \infty} X^\lambda = |X^1|$,
  which implies that $\lim_{\lambda \to \infty} P(\lambda) = P(1)$.
\end{enumerate}

For small values of $\lambda$, the first effect dominates, while for
large values of $\lambda$, the second effect does, as can be observed
in Fig.~\ref{fig:contre-ex}. Thus, for $\lambda$ large enough, the
containment probability decreases with $\lambda$, which contradicts
the intuition of \eqref{eq:comp}.

In the next section, we shall define a suitable dominance
relationship between the pull strengths, that is, one that implies an
order relationship between the containment probabilities.

\section{Comparison Theorem under symmetric dominance assumption}
\label{sec:main}

\subsection{Comparison Theorem in dimension $d=1$}
\label{sec:dim1}

To avoid the phenomenon of concentration in ``safe havens'' where
exits are subsequently easier, one can ``symmetrize'' the dominance
assumption of \eqref{eq:assum-orig} as in the following Theorem (see
Fig.~\ref{fig:dominance} for illustration).

\begin{figure}[htp]
\centering
\includegraphics[width=6cm]{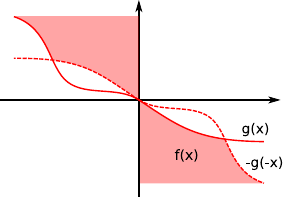}
\caption{Symmetrization of the dominance assumption.}
\label{fig:dominance}
\end{figure}

\begin{theorem}[Comparison Theorem in dimension 1]
  \label{thm:onedimdiffusion}
  Let $f,g$ be continuous functions satisfying
  \begin{eqnarray}
    \label{eq:assump-sym}
    \forall x \in [0,K],&  f(x) \leq \min(g(x),-g(-x))\nonumber\\
    \forall x \in [-K,0],&  f(x) \geq \max(g(x),-g(-x)),  \end{eqnarray}
  then, given the SDEs
  \begin{eqnarray}
    dX_t &=& f(X_t) dt + \sgn(X_t) dB_t\nonumber\\
    dY_t &=& g(Y_t) dt + \sgn(Y_t) dB_t,\nonumber
  \end{eqnarray}
  one has almost surely
  $|X_t| \leq |Y_t|$ for all $t \leq T_K := \inf\{ t > 0 : |Y_t|=K\}$.
  In particular, one has
  \[
    \forall R \in (0,K],\ \forall T \geq 0,\ \P( \sup_{t \leq T} |X_t|
    \leq R ) \geq \P(\sup_{t \leq T} |Y_t| \leq R). 
    \]
\end{theorem}

Note that the processes $X$ and $Y$ and defined using the \emph{same}
noise process~$B$ -- a technique called
\emph{coupling}~\cite{lindvall2002lectures}.  Heuristically, the
coupling is as follows
\begin{itemize}
\item if $X_t$ and $Y_t$ have the same sign, then $X_{t+dt}$ and
  $Y_{t+dt}$ are constructed using the same noise;
\item if $X_{t}$ and $Y_t$ have opposite sign, then $X_{t+dt}$ and
  $Y_{t+dt}$ are constructed using opposite noises.
\end{itemize}
Thus, in both cases, the \emph{absolute values} of $X_t$ and $Y_t$
move in the same direction.

The main difficulty in proving Theorem~\ref{thm:onedimdiffusion} is
the time instants when $X$ and $Y$ go close to 0. We shall tackle this
difficulty by an approximation technique. Consider the stronger
assumptions of the following Lemma.

%

\begin{lemma}
  \label{lem:nonsatifsfying}
  Let $f,g$ be continuous functions satisfying
  \eqref{eq:assump-sym}. Assume moreover that there exists
  $\epsilon>0$ such that
  \begin{equation}
    \label{eq:nonsatisfying}
    \forall |x| \leq \epsilon,\ f(x) = g(x).
  \end{equation}
  Then the conclusions of Theorem~\ref{thm:onedimdiffusion} hold.
\end{lemma}

\begin{proof}
  Note first that \eqref{eq:nonsatisfying} and \eqref{eq:assump-sym}
  together imply that 
  \[
  \forall |x| \leq \epsilon,\ -f(-x) = f(x) = g(x) = -g(-x).
  \]

  Let $B$ be a standard Brownian motion, we construct on the same
  probability space two processes $X$ and $Y$ satisfying
  \begin{eqnarray}
    X_0 &=& Y_0 = 0 \nonumber\\
    \label{eq:integral-X}
    X_t &=& \int_0^t f(X_s) ds + \int_0^t\sgn(X_s) dB_s
    \\ 
    \label{eq:integral-Y}
    Y_t &=& \int_0^t
            g(Y_s) ds + \int_0^t\sgn(Y_s) dB_s.
  \end{eqnarray}
  Note that there is no trajectorial uniqueness for this type of
  stochastic differential equation. However, there exist solutions,
  and uniqueness in law holds by boundedness of $f$ and $g$ on
  compact intervals.

  Given a solution $Y$ of \eqref{eq:integral-Y}, define
  a sequence of hitting times as follows. Let
  $\tau^{(0)} := 0$ and
  \[
  \upsilon^{(k)} := \inf\{ t > \tau^{(k)} : |Y_t| = \epsilon\}; \quad
  \tau^{(k+1)} := \inf\{ t > \upsilon^{(k)} : Y_t = 0\}.
  \]
  In other words, $\upsilon^{(k)}$ is the first time, after time
  $\tau^{(k)}$, that $Y$ leaves the strip $[-\epsilon,\epsilon]$, and
  $\tau^{(k+1)}$ is the first time, after time $\upsilon^{(k)}$, that
  $Y$ hits $0$.

  We now construct $X$ as follows (see Fig.~\ref{fig:traj} for sample
  paths of $X$ and $Y$)
  \begin{itemize}
  \item For all $t \in [\tau^{(k)},\upsilon^{(k)}]$, $X_t=Y_t$;
  \item On $[\upsilon^{(k)},\tau^{(k+1)}]$, $X$ is the unique solution of
    \eqref{eq:integral-X} on this interval. 
  \end{itemize}
  
  \begin{figure}[htp]
    \centering
    \includegraphics[width=10cm]{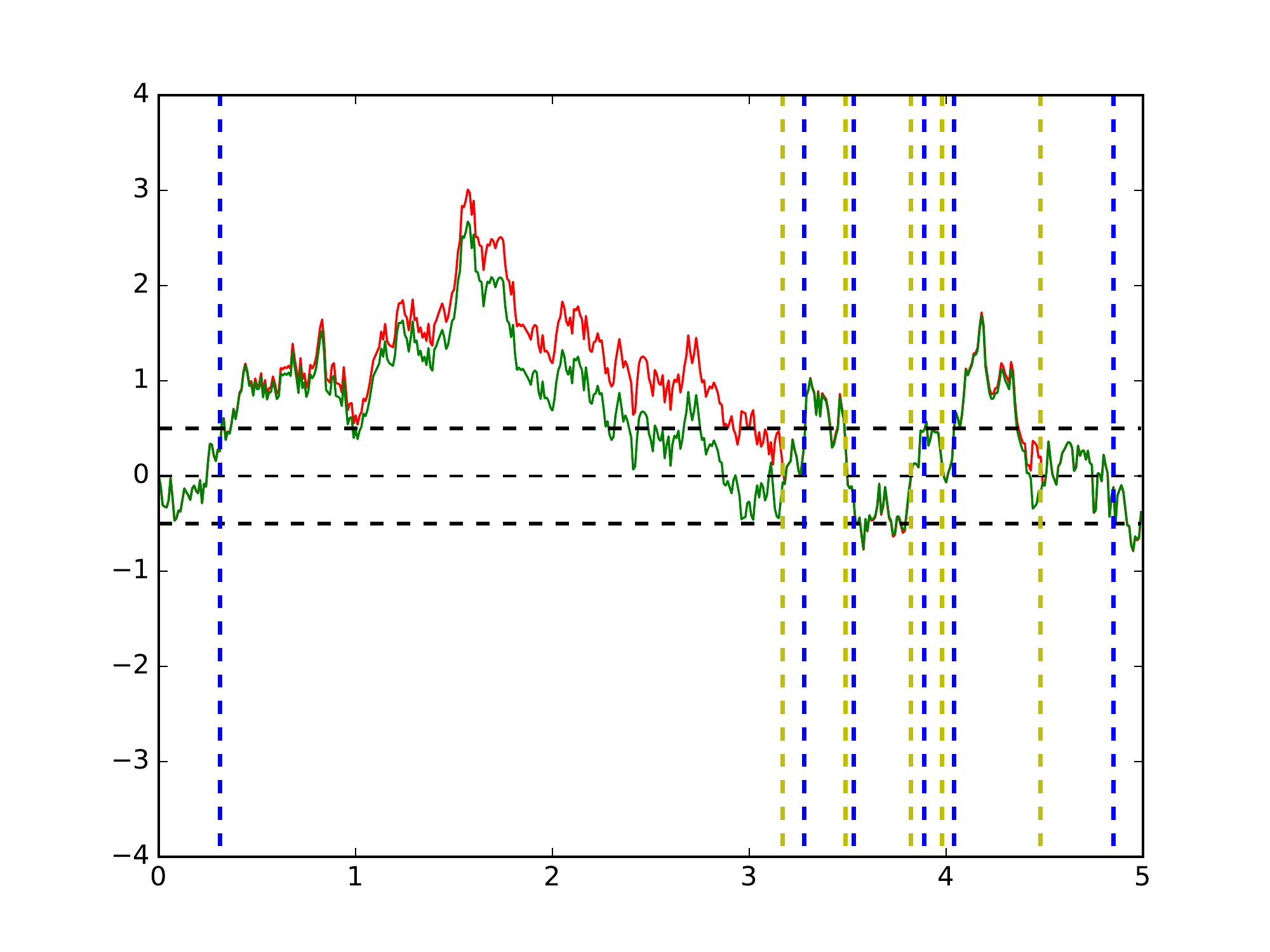}
    \caption{Sample trajectories for the processes $X$ (green) and $Y$
      (red). Here $\epsilon=0.5$ (black dashed lines). The hitting
      times $\tau$ (dashed yellow) and $\upsilon$ (dashed blue) are also
      indicated. Note that $X=Y$ in the intervals [yellow, blue] and
      $|X| \leq |Y|$ in the intervals [blue, yellow].}
    \label{fig:traj}
  \end{figure}

  We show, by induction on $k\in\mathbb{N}$, the following properties
  of $X$
  \begin{description}
  \item[(i)] $X$ is a solution of \eqref{eq:integral-X} on
    $[\tau^{(k)},\upsilon^{(k)}]$ and $[\upsilon^{(k)},\tau^{(k+1)}]$;
  \item[(ii)] $|X|\leq |Y|$ in the above intervals;
  \item[(iii)] $X_{\tau^{(k+1)}}=0$.
  \end{description}

  \emph{Initialization:} \textbf{(i)} For $k=0$, note that, for all
  $t\in[\tau^{(0)},\upsilon^{(0)}]$, $X_t=Y_t\leq \epsilon$. Thus, by
  \eqref{eq:nonsatisfying}, one has $f(X_t)=g(Y_t)$ over the whole
  interval, which in turn implies that $X=Y$ is a solution of
  \eqref{eq:integral-X} on the interval. The fact that $X$ is a
  solution of \eqref{eq:integral-X} on $[\upsilon^{(0)},\tau^{(1)}]$ is
  by construction.

  \textbf{(ii)} Next, $|X_t|\leq |Y_t|$ trivially on
  $[\tau^{(0)},\upsilon^{(0)}]$. On $[\upsilon^{(0)},\tau^{(1)}]$, one
  has, by It\^o-Tanaka formula
  \begin{eqnarray*}
    |X_t| &=& \int_0^t \sgn(X_s) f(X_s) ds + B_t + L^X_t\\ 
    |Y_t| &=& \int_0^t \sgn(Y_s) g(Y_s) ds + B_t + L^Y_t,
  \end{eqnarray*}
  where $L^X$ and $L^Y$ are the local times at $0$ of $X$
and $Y$ respectively. This yields
  \begin{equation}
    \label{eq:auxDifferences}
    |Y_t| - |X_t| = \int_0^t \sgn(Y_s) g(Y_s) - \sgn(X_s) f(X_s) ds +
    L^Y_t - L^X_t.
  \end{equation}

  Assume by contradiction that there exists some time
  $t\in [\upsilon^{(0)},\tau^{(1)})$ such that $|X_t|=|Y_t|$, and
  $|X_{t+u}|>|Y_{t+u}|>0$ for all $u>0$ small enough.  As both
  $|X_{t+u}|$ and $|Y_{t+u}|$ are different from $0$ for $u$ close
  enough to $0$, one has that $L^X$ and $L^Y$ are constant in a
  neighbourhood of $t$. Therefore, by \eqref{eq:auxDifferences},
  $h: t \mapsto |Y_t|-|X_t|$ is differentiable in a neighborhood of
  $t$, and its derivative is, in that neighborhood,
  \[
  h'(s) = \sgn(Y_s) g(Y_s) - \sgn(X_s) f(X_s) \geq 0, 
  \]
  which contradicts the fact that $|X_{t+u}| > |Y_{t+u}|$ for $u$
  small enough.

  We have thus shown that  $|X_t|\leq |Y_t|$ on
  $[\upsilon^{(0)},\tau^{(1)})$. The inequality can be extended to the
  closed interval by continuity.

  \textbf{(iii)} Since $|X|\leq |Y|$ on the interval
  $[\upsilon^{(0)},\tau^{(1)}]$ and that $Y_{\tau^{(1)}}=0$, one has
  $X_{\tau^{(1)}}=0$. 

  \emph{Induction:} By the induction hypothesis, one has
  $X_{\tau^{(k)}}=Y_{\tau^{(k)}}=0$. Thus the proof that
  $|X_s| \leq |Y_s|$ for $s \in [\tau^{(k)},\tau^{(k+1)}]$ is
  straightforward by using the Markov property and adapting the
  \emph{Initialization} step.

  To complete the proof, note that
  $\lim_{k \to \infty}\tau^{(k)} = \infty$, which can be obtained by
  observing that, by the law of large numbers, one has almost surely
  \[
  \lim_{k \to \infty} \frac{1}{k} \sum_{j=1}^k \upsilon^{(j)} -   \tau^{(j)} =
  \E(\upsilon^{(1)}) > 0.
  \]
\end{proof}

We are now in a position to prove Theorem~\ref{thm:onedimdiffusion}.

\begin{proof}
  Let $f$ and $g$ be two functions satisfying
  \eqref{eq:assump-sym}. For all $\epsilon > 0$, we set
  \begin{eqnarray}
    f^\epsilon(x) &:=& \ind{x>2\epsilon} f(x-2\epsilon) + \ind{x <
                       -2\epsilon} f(x+2\epsilon)\nonumber\\
    g^\epsilon(x) &:=&
                       \ind{x>\epsilon} g(x-\epsilon) + \ind{x < -\epsilon} g(x+\epsilon).    \nonumber
  \end{eqnarray}
  Note that $f^\epsilon,g^\epsilon$ satisfy \eqref{eq:assump-sym} and
  \eqref{eq:nonsatisfying}, therefore, by
  Lemma~\ref{lem:nonsatifsfying}, the conclusions of
  Theorem~\ref{thm:onedimdiffusion} hold for $X^\epsilon$ and
  $Y^\epsilon$. Next, since the law of $(X^\epsilon, Y^\epsilon)$
  converges toward the law of $(X,Y)$, one can conclude, by
  Skorokhod's embedding theorem, that there is trajectorial
  convergence, and that the limiting process $(X,Y)$ satisfies the
  conclusion of Theorem~\ref{thm:onedimdiffusion}.
\end{proof}

\begin{remark}
  If $Y$ is an Ornstein-Uhlenbeck process with pull strength $\lambda$, then
  condition~\eqref{eq:assump-sym} becomes
  \[
  \forall x \in [-K,K],\  \sgn(x)f(x) \leq -\lambda |x|.
  \]
\end{remark}

\subsection{Comparison Theorem in dimension $d\geq 2$}
\label{sec:dimd}

To enforce a ``symmetric'' dominance assumption in dimension
$d\geq 2$, we introduce the radial decomposition: for
$x\in \R^d \backslash \{0\}$, denote by $R_x$ the matrix of the
rotation that brings $\frac{x}{\|x\|}$ to the first basis vector $e_1$
(by convention, $R_0 :=
\mathrm{Id}$).
Theorem~\ref{thm:onedimdiffusion} can now be extended to dimension
$d\geq 2$ as follows.

\begin{theorem}[Comparison Theorem in dimension $d\geq 2$]
  \label{thm:dimd}
  Let $f,g$ be continuous functions satisfying
  \begin{equation}
    \label{eq:conddimd}
    \forall r \in \R_+,\ \sup_{\theta \in \mathbb{S}_{d-1}}
    f(r\theta)^\top\theta \leq \inf_{\theta \in \mathbb{S}_{d-1}}
    g(r\theta)^\top\theta.
  \end{equation}
  then, given the SDEs
  \begin{eqnarray}
    dX_t &=& f(X_t) dt + R_{X_t}^{-1} dB_t \nonumber\\ 
    dY_t &=& g(Y_t) dt + R_{Y_t}^{-1} dB_t, \nonumber
  \end{eqnarray}
  one has almost surely
  $\|X_t\| \leq \|Y_t\|$ for all $t \leq T_K := \inf\{ t > 0 : \|Y_t\|=K\}$.
  In particular, one has
  \[
    \forall R \in (0,K],\ \forall T \geq 0,\ \P( \sup_{t \leq T} \|X_t\|
    \leq R ) \geq \P(\sup_{t \leq T} \|Y_t\| \leq R). 
    \]
\end{theorem}

\begin{proof}
  We follow the same strategy as in the proof of
  Theorem~\ref{thm:dimd}. Note that, using It\^o-Tanaka formula in
  dimension $d\geq 2$, one has
  \[
  ||X_t|| = \int_0^t f(X_s)^\top \frac{X_s}{||X_s||} ds + B_t.e_1 + L^X_t,
  \]
  and a similar formula for $Y$, where 
  \[
  L^X_t := \lim_{\epsilon \to 0}
  \frac{1}{V_d\epsilon^d} \int_0^t \ind{||X_s|| \leq \epsilon} ds.
  \]
  The rest of the proof is similar to that of Theorem~\ref{thm:dimd}.
\end{proof}

\begin{remark}
  \label{rem:dimd}
  If $Y$ is an Ornstein-Uhlenbeck process with pull strength $\lambda$, then
  condition~\eqref{eq:conddimd} becomes
  \begin{eqnarray}
    \forall r \in \R_+,\ \sup_{\theta \in \mathbb{S}_{d-1}}
    f(r\theta)^\top\theta \leq -\lambda r\nonumber.
  \end{eqnarray}
\end{remark}

\section{Link with contraction theory}
\label{sec:contraction}

Contraction theory~\cite{lohmiller1998contraction} provides a set of
tools to analyze the exponential stability of nonlinear systems, and
has been applied notably to observer design (see
e.g.~\cite{bonnabel2014contraction}), synchronization analysis (see
e.g.~\cite{tabareau2010synchronization}), and systems
neuroscience. Nonlinear contracting systems enjoy desirable
aggregation properties, in that contraction is preserved under many
types of system combinations given suitable simple
conditions~\cite{lohmiller1998contraction}.

We say that $f$ is \emph{contracting} with contraction rate $\lambda>0$
in the identity  metric~\cite{lohmiller1998contraction} if
\[
\forall x,t, \ \lambda_{\max}\left(\frac{\partial f}{\partial
  x}\right) \leq -\lambda,
\]
where $\lambda_{\max}(A)$ denotes the largest eigenvalue of the
symmetric part of matrix~$A$. A specialized result of contraction
theory is that, if $f$ is contracting with contraction rate $\lambda$,
then all system trajectories converge exponentially to a single
trajectory, with convergence rate $\lambda$. More general settings of
contraction theory can cater for the dependency of $f$ on the time
parameter $t$ as well as nonlinear metrics. For simplicity, however,
our current discussion is carried out without the dependency of $f$ on
$t$ and in the identity metric. Including time-dependency could be
addressed by adapting condition \eqref{eq:conddimd} to include
uniformity over $t\in\mathbb{R}$ and $\omega\in\Omega$. Extension to
nonlinear metrics would likely involve checking whether the metrics
are compatible with the symmetric dominance assumption. Such
extensions will be investigated in our future work.

Assume now that $f$ is contracting with contraction rate $\lambda$ in
the identity metric and consider two $d$-dimensional SDEs
\begin{eqnarray*}
  dX_t&=&f(X_t)dt+\sigma dB^X_t\\ 
  dY_t&=&f(Y_t)dt+\sigma dB^Y_t.
\end{eqnarray*}

Consider the $d$-dimensional process $Z:=Y-X$. One has 
\[
dZ_t = dY_t-dX_t = [f(Y_t)-f(X_t)]dt + \sqrt{2}\sigma dB^Z_t.
\]

Since $f$ is smooth, one can write\,\footnote{See for instance
  Lemma~1 at \url{https://en.wikipedia.org/wiki/Mean_value_theorem}.}
\begin{eqnarray*}
  f(Y_t)-f(X_t) &=& \left(\int_0^1 \frac{\partial f}{\partial
                    x} (X_t+s(Y_t-X_t)) ds
                    \right)(Y_t-X_t)\\ 
                &=:&g(X_t,Y_t) Z_t.
\end{eqnarray*}

Thus, one can rewrite
\[
dZ_t = g(X_t,Y_t)Z_t dt + \sqrt{2}\sigma dB^Z_t,
\]
where $g(X_t,Y_t)$ can be seen as an external driving signal.

Define $h(z):=g(x,y)z$. Let us evaluate the radial component
of $h$. For that, set $z=r\theta$ where $r>0$,
$\theta\in\mathbb{S}_{d-1}$. One has
\begin{eqnarray}
  h(r\theta)^\top\theta &=& \frac{1}{r}(g(x,y)z)^\top z  \nonumber \\
                        &=& \frac{1}{r}z^\top\left(\int_0^1 \frac{\partial f}{\partial
                            x} (x+s(y-x)) ds
                            \right)^\top z \nonumber\\
                        &=&\frac{1}{r}\int_0^1 z^\top\left(\frac{\partial f}{\partial
                           x} (x+s(y-x))\right)^\top z ds \nonumber\\
                        &\leq& \frac{1}{r} \int_0^1 -\lambda \|z\|^2 ds=
                               -\lambda r.\nonumber
\end{eqnarray}
Since $-\lambda r$ is the radial component of a $d$-dimensional
Ornstein-Uhlenbeck process with pull strength $\lambda$ (see
Remark~\ref{rem:dimd}), Theorem~\ref{thm:dimd} can be used to bound
the containment probabilities for the distance $\P( \sup_{t \leq T}
\|Y_t-X_t\| \leq R )$ by the corresponding containment probabilities
of a $d$-dimensional Ornstein-Uhlenbeck process with pull
strength~$\lambda$ and noise strength $\sqrt{2}\sigma$.

\section{Conclusion}
\label{sec:conclusion}

We have defined a dominance relationship between the pull strengths of
two nonlinear stochastic systems that implies an order relationship
between their containment probabilities. This result enables
establishing tight bounds on the containment probabilities for a large
class of nonlinear systems by comparing them with suitable
Ornstein-Uhlenbeck processes, for which containment probabilities can
be numerically obtained.

One important implication of this result is that one can immediately
bound the containment probabilities of stochastic systems that are
contracting with rate $\lambda$ by those of Ornstein-Uhlenbeck
processes with pull strength $\lambda$.

The results presented here may have many exciting applications in
control theory. For example, the design of controllers for optical
manipulation in~\cite{yan2016stochastic} could be extended to deal
with nonlinear trapping forces. Another application could be to
develop a rigorous theory of stability for Extended Kalman Filters,
e.g. by extending the contraction-theory-based analysis
of~\cite{bonnabel2014contraction} to stochastic systems. Yet another
avenue would be to establish tight bounds on the time taken by
stochastic optimization algorithms -- such as the Stochastic Gradient
Descent widely used in machine learning -- to escape local
minima~\cite{kleinberg2018alternative}. Exploring such applications is
the subject of ongoing research.

\bibliographystyle{abbrv}
\bibliography{cri}

\begin{thebibliography}{10}

\bibitem{ashkin1986observation}
A.~Ashkin, J.~M. Dziedzic, J.~Bjorkholm, and S.~Chu.
\newblock Observation of a single-beam gradient force optical trap for
  dielectric particles.
\newblock {\em Optics letters}, 11(5):288--290, 1986.

\bibitem{bonnabel2014contraction}
S.~Bonnabel and J.-J. Slotine.
\newblock A contraction theory-based analysis of the stability of the
  deterministic extended kalman filter.
\newblock {\em IEEE Transactions on Automatic Control}, 60(2):565--569, 2014.

\bibitem{finch2004ornstein}
S.~Finch.
\newblock Ornstein-uhlenbeck process.
\newblock
  \url{http://citeseerx.ist.psu.edu/viewdoc/summary?doi=10.1.1.710.4200}, 2004.
\newblock Accessed: 2018-10-05.

\bibitem{karvonen2018stability}
T.~Karvonen, S.~Bonnabel, E.~Moulines, and S.~S{\"a}rkk{\"a}.
\newblock On stability of a class of filters for non-linear stochastic systems.
\newblock {\em arXiv preprint arXiv:1809.05667}, 2018.

\bibitem{kleinberg2018alternative}
R.~Kleinberg, Y.~Li, and Y.~Yuan.
\newblock An alternative view: When does sgd escape local minima?
\newblock {\em arXiv preprint arXiv:1802.06175}, 2018.

\bibitem{kushner1967stochastic}
H.~Kushner.
\newblock {\em Stochastic Stability and Control}.
\newblock Academic Press, 1967.

\bibitem{lindvall2002lectures}
T.~Lindvall.
\newblock {\em Lectures on the coupling method}.
\newblock Courier Corporation, 2002.

\bibitem{lohmiller1998contraction}
W.~Lohmiller and J.-J.~E. Slotine.
\newblock On contraction analysis for non-linear systems.
\newblock {\em Automatica}, 34(6):683--696, 1998.

\bibitem{pham2009contraction}
Q.-C. Pham, N.~Tabareau, and J.-J. Slotine.
\newblock A contraction theory approach to stochastic incremental stability.
\newblock {\em IEEE Transactions on Automatic Control}, 54(4):816--820, 2009.

\bibitem{tabareau2010synchronization}
N.~Tabareau, J.-J. Slotine, and Q.-C. Pham.
\newblock How synchronization protects from noise.
\newblock {\em PLoS computational biology}, 6(1):e1000637, 2010.

\bibitem{yan2016stochastic}
X.~Yan, C.~C. Cheah, Q.~M. Ta, and Q.-C. Pham.
\newblock Stochastic dynamic trapping in robotic manipulation of micro-objects
  using optical tweezers.
\newblock {\em IEEE Transactions on Robotics}, 32(3):499--512, 2016.

\end{thebibliography}

\end{document}